\documentclass[10pt, reqno]{amsart}

\usepackage{amsmath}
\usepackage{amsfonts}
\usepackage{amssymb}

\usepackage{amssymb,latexsym,amsmath,amsfonts}
\usepackage{hyperref}
\usepackage{amsthm}
\usepackage{dsfont}
\usepackage[mathscr]{eucal}
\usepackage[english]{babel}
\usepackage{lipsum}

\theoremstyle{definition}
\newtheorem*{defn}{Definition}
\newtheorem{thm}{Theorem}[section]
\newtheorem{lem}{ \bf Lemma}[section]
\newtheorem{prop}{\bf Proposition}[section]
\newtheorem{cor}{\bf Corollary}[section]
\theoremstyle{remark}
\newtheorem{rem}{Remark}[section]
\theoremstyle{example}
\newtheorem{ex}{Example}[section]
\newcommand{\be}{\begin{equation}}
\newcommand{\ee}{\end{equation}}
\newcommand{\Bea}{\begin{eqnarray*}}
\newcommand{\Eea}{\end{eqnarray*}}
\newcommand{\bea}{\begin{eqnarray}}
\newcommand{\eea}{\end{eqnarray}}

\numberwithin{equation}{section}

\begin{document}
\title[Volume Growth on manifolds with more than one end]{Volume Growth on manifolds with more than one end}
\author{Anushree Das \and Soma Maity}

\address{Department of Mathematical Sciences, Indian Institute of Science Education and Research Mohali, \newline Sector 81, SAS Nagar, Punjab- 140306, India.}
\email{somamaity@iisermohali.ac.in}

\address{Department of Mathematical Sciences, Indian Institute of Science Education and Research Mohali, \newline Sector 81, SAS Nagar, Punjab- 140306, India.}
\email{ph20016@iisermohali.ac.in}


\begin{abstract} 
For an open manifold $M$ and a function $v$ with bounded growth of derivative, there exists a Riemannian metric of bounded geometry on $M$ such that the volume growth function lies in the same growth class as $v$. This was proved by R. Grimaldi and P. Pansu with the proof focusing on the case of manifolds with a single end. We prove this in the case of manifolds with multiple ends and call the constructed metrics Grimaldi-Pansu metrics. We give uniform bounds for the volume growth function of these metrics in terms of the given bgd-function in the case of a certain class of manifolds which can be written as connected sums of a finite collection of closed and compact manifolds. We study the volume doubling condition and the Relatively Connected Annulus (R.C.A.) property of the Grimaldi-Pansu metrics, which play an important role in studying geometric analysis on manifolds with finitely many ends.

\end{abstract}

\footnotetext{Data sharing is not applicable to this article as no datasets were generated or analyzed during the current study.}

\maketitle

\section{Introduction}\label{intro} Let $(M,g)$ be a complete non-compact smooth Riemannian manifold of dimension $m$ without boundary. $(M,g)$ is said to have bounded geometry if the injectivity radius $i_g$ is bounded below by $1$ and the sectional curvature $K$ satisfies $|K|\leq 1$. Any non-compact Riemannian covering space of a compact Riemannian manifold satisfies the bounded geometry condition. A volume growth function $v$ of $(M,g)$ is a positive function on $\mathbb{R}_+$ such that $v(r)$ is the volume of the ball of radius $r$ centered at a point $o$ in $M$. Badura, Funar, Grimaldi, and Pansu investigate those functions which are volume growth functions of a Riemannian manifold with bounded geometry, and their relations with the topology of the manifold in \cite{GP}, \cite{BM}, \cite{FG}. 

    A function $v:\mathbb{N} \to \mathbb{R}_+$ is said to have bounded growth of derivative if there exists a positive integer $L$ such that, $\forall n\in \mathbb{N}$,
    $$\frac{1}{L}\leq v(n+2)-v(n+1)\leq L(v(n+1)-v(n)).$$

We call a function with bounded growth of derivative a \emph{bgd-function} in short. A volume growth function of a Riemannian manifold with bounded geometry, when restricted to $\mathbb{N}$, is a bgd-function \cite{GP}.
 Two non-decreasing functions $f,h:\mathbb{N}\rightarrow \mathbb{R}_{+}$ have the same growth type if there exists an integer $A\geq 1$ such that for all $n\in \mathbb{N}$, 
\Bea
f(n)\leq Ah(An+A)+A \ \ \text{and} \ \ h(n)\leq Af(An+A)+A. 
\Eea
We call the constant $A$ a \textit{growth constant} for $f$ and $h$. Functions of the same growth type define an equivalence relation on the set of non-decreasing functions from $\mathbb{N}$ to $\mathbb{R}_{+}$, and the equivalence classes are denoted by $[.].$ The volume growth functions of a Riemannian manifold based at two different points are of the same growth type and hence belong to the same equivalence class. Any scaling of the Riemannian metric also does not change the growth type of volume growth functions. 
Next, we define the ends of a manifold.
\begin{defn}
    Consider an exhaustion $K_1\subseteq K_2\subseteq\dots$ of $M$ by compact submanifolds. Let $U_i=M\setminus K_i$. An end of a manifold is any sequence $\{V_i\}_{i=1}^\infty$ where each $V_i$ is a connected component of $U_i$ and $V_1\supseteq V_2\supseteq\cdots$.
\end{defn}
The number of ends of a manifold does not depend on the choice of the exhaustion, but only on the topology of $M$. A non-compact manifold is said to be of finite topological type if it admits an exhaustion by compact submanifolds $M_i$ such that the boundary of $M_i$, denoted by $\partial M_i$, are all diffeomorphic. Grimaldi and Pansu constructed a Riemannian metric of bounded geometry on an open manifold such that the volume growth lies in a given class of bgd-functions in \cite{GP}.
\begin{thm}
\label{t3}
Let $M$ be a non-compact connected manifold without boundary.
\begin{enumerate}
    \item If $M$ has finite topological type, every bgd-function belongs to the growth type of a Riemannian manifold of bounded geometry diffeomorphic to $M$.
    \item If $M$ has infinite topological type, a bgd-function $v$ belongs to the growth type of a Riemannian manifold of bounded geometry diffeomorphic to $M$ if and only if $\lim_{n\to\infty} \frac{v(n)}{n}=+\infty$.
\end{enumerate}
\end{thm}
Though the above theorem is stated in \cite{GP} in full generality, the proof that the volume growth of the metric constructed there belongs to a given class holds only for manifolds with one end. We discuss this in more detail towards the end of this section. In this paper, we give a detailed proof of Theorem \ref{t3} for manifolds with multiple ends following the same idea, and the Riemannian metrics constructed thereof are called Grimaldi-Pansu metrics. We also show that the growth constant depends only on the topology of the manifold and on the given growth class.

\subsection*{Connected sum along a tree :}\label{connectedsum} All trees are locally finite in this paper. Let $\mathcal{U}$ be a finite collection of $n$-dimensional closed manifolds. We consider an open manifold $M$ which is a connected sum of manifolds from $\mathcal{U}$ along an infinite tree $T$, i.e. the compact manifolds are placed on the vertices of the tree, and an edge represents a connected sum between the pieces on two vertices. 

In \cite{BBM0} Bessières, Besson, and Maillot showed that if the scalar curvature of a complete Riemannian manifold with bounded geometry is uniformly bounded below by a positive constant then the manifold is diffeomorphic to a manifold that is a connected sum along a locally finite graph of finitely many spherical manifolds. In this paper, we prove the following theorem.

\begin{thm}\label{graph}
    Let $M$ be a connected open manifold which is a connected sum of manifolds from a finite set of closed manifolds $\mathcal{U}$ along a rooted infinite tree $T$. Given a bgd-function $v$, there exists a Grimaldi-Pansu metric $g$ on $M$ such that the growth constant of the volume growth function of $g$ based at the root of $T$ and $v$ is bounded above by a constant depending only on $T$, $v$, and $\mathcal{U}$. Moreover, if the degree of each vertex of $T$ is bounded by $k$ then the growth constant depends only on $k$, $v$, and, $\mathcal{U}$. 
\end{thm}

Next, we focus on manifolds with finitely many ends and study the existence of Grimaldi-Pansu metrics that satisfy the R.C.A. or R.C.E. conditions while having their volume growth in a desired growth class. In Theorem \ref{rca1} we prove the existence of such metrics for manifolds that are connected sums along a tree, and in Theorem \ref{rca2} we prove the existence of such metrics in a general one-ended manifold under some additional constraints on the given bgd-function. We observe that the volume growth of a Grimaldi-Pansu metric depends on the base point crucially and does not satisfy homogeneity conditions of the growth of volume, like the volume doubling condition. Gilles Carron showed that if a manifold admits a metric that satisfies the volume doubling condition then it has finitely many ends in \cite{CG}. Polynomials are bgd-functions that satisfy the doubling conditions. If $M$ has infinitely many ends then a Grimaldi-Pansu metric on $M$ in the class of a polynomial does not satisfy any doubling condition. We show that if a Grimaldi-Pansu metric satisfies a doubling condition, its volume growth is at most linear in Theorem \ref{doubling}. It would be interesting to find a metric on a manifold with finitely many ends such that its volume growth function belongs to the class of a given a bgd-function that satisfies the doubling condition. The R.C.A. or R.C.E. and volume doubling conditions have implications in geometric analysis on the ends of manifolds and have been extensively used by Grigor'yan, Ishiwata, Saloff-Coste, and Gilles Carron among others in \cite{CG}, \cite{GSC},  \cite{GSC2}, \cite{PG} to study estimates on heat kernel, parabolic Harnack inequality, euclidean volume growth on manifolds with finitely many ends. 

Section \ref{s2} briefly discusses the construction of a Grimaldi-Pansu metric on a manifold given in \cite{GP}. To define a metric of a certain volume growth the authors consider an exhaustion $\{K_i\}$ of the manifold, then define a metric on each $Q_i=K_{i+1}\setminus K_i$. A point $y_{Q_i}$ is fixed in each $\partial Q_i$. To prove that the volume growth of the metric lies in a given equivalence class, the authors use the points $y_{Q_i}$ to create a path between the basepoint and any other point $x\in M$. This path is then used to estimate the distance of $x$ from the basepoint. Here, the connectedness of each $Q_i$ is implicitly assumed. In the case of a manifold with multiple ends, for any exhaustion $\mathcal{A}_i$, there exists some $i$ such that for all $j>i$, $Q_j$ has multiple components. Then this algorithm chooses a point $y_{Q_i}$ in only one of those components. Hence for any point $x$ lying in a different component of $Q_i$, this process does not estimate the distance of $x$ from the base point. Therefore it is necessary to choose different points $y_{Q_i}$ in each component of $Q_i$, and to do so, we need to take into account the number of components of $Q_i$ for each $i$, and hence, the number of ends of $M$. We have modified the proof in \cite{GP} to account for this situation to prove Theorem \ref{t3}. Some propositions, and functions used in \cite{GP} have been modified slightly with this condition in mind. 

In Section \ref{s3}, we focus on open manifolds that are connected sums of finitely many closed manifolds along a tree. Generalizing the Kneser-Milnor Prime Decomposition Theorem Besson, Bessières, and Maillot classified open 3-manifolds that are connected sum of finitely many closed prime manifolds in \cite{BBM}. They showed that for any given tree $T$ and a finite set $\mathcal{U}$ of closed 3-manifolds, the manifolds resulting from connected sums of elements of $\mathcal{U}$ along $T$ with different choices of representative elements at the vertices are not necessarily diffeomorphic. If the number of ends of $T$ is infinite and the cardinality of $\mathcal{U}\geq 3$ then the number of such manifolds which are not diffeomorphic to each other is infinite. We use the ideas of Grimaldi and Pansu to prove that for a given finite collection of closed $n$-manifolds, and any bgd-function $v$, different manifolds that are obtained from connected sum compositions via the same tree $T$, admit a metric with volume growth function in the growth class of $v$ such that the growth constant depends only on $T$, $v$, and, $\mathcal{U}$. If the degree of each vertex is bounded by $k$ then the growth constant is independent of $T.$
 
\subsection*{Acknowlegdement}
The authors would like to thank Gérard Besson and Harish Seshadri for their comments. The first author is supported via a research grant from the National Board of Higher Mathematics, India.

\section{Proof of Grimaldi-Pansu's theorem for manifolds with multiple ends}\label{s2}

We first summarise the construction of the required Riemannian metric on a one-ended manifold by Grimaldi and Pansu in \cite{GP}. Given a one-ended manifold $M$ with dimension $m$ and a bgd-function $v$ they construct a metric of bounded geometry on $M$ such that the volume growth function has the same growth type as $v$. We shall be modifying their construction and proof in this section.

Grimaldi and Pansu show in Lemma $11$ of \cite{GP} that it is enough to consider functions $v$ that satisfy the following conditions, since for any given bgd-function we can find a function satisfying these conditions and lying in the same growth class as the given function:
\begin{enumerate}
    \item $v(0)=1.$
    \item For all $n\in N, \ \ 2\leq v(n+2)-v(n+1)\leq 2(v(n+1)-v(n)).$
    \item $v(n)=O(\lambda^n)$ for some $\lambda<2.$
\end{enumerate}

Consider an exhaustion $\{\mathcal{A}_j\}$ of $M$ such that each $\mathcal{A}_{j+1}\setminus \mathcal{A}_j$ is a manifold of dimension $m$ with boundary. Let $Q_j=\mathcal{A}_{j+1}\setminus \mathcal{A}_j$. The boundary components of $Q_j$ are denoted by $\partial^+ Q_j$ and $\partial^- Q_j$, where $\partial^+ Q_j$ is diffeomorphic to $\partial^- Q_{j+1}$ by the identity diffeomorphism. 

Fix a subset $S$ of $\mathbb{N}$ such that $\liminf_{n\rightarrow \infty}\frac{S\cap \{0,1,\dots,n\}}{n}=0$. Consider a rooted tree $\mathcal{T}$ with growth $v$ such that $\mathcal{T}$ has a single trunk, each vertex has at most $2$ branches, and $S$ is the set of those vertices of the trunk of $\mathcal{T}$ which only have one branch. We refer to Lemma $10$ in \cite{GP} for a proof of the existence of such a tree. Write $S$ as $S=\bigcup_j [n_j,n_j+t_j-1]$. Attach a single copy of $Q_j$ collectively to all the vertices in the interval $[n_j,n_j+t_j-1]$.
Additional pieces attached to the tree are the following. The piece $R_i$ diffeomorphic to $\partial^+Q_j\times [0,1]$ with a disc removed; $K$, an $m$-dimensional cylinder diffeomorphic to $S^{m-1}\times [0,1]$; $J$, diffeomorphic to an $m$-dimensional sphere with $3$ balls removed; $HS$, a half-sphere of dimension $m$.  Attach the pieces $R_j$ to each vertex $k$ of the trunk where $n_j+t_j\leq k < n_{j+1}$, that is, the vertices of the trunk with $2$ branches. For the vertices which do not lie on the trunk, other pieces are attached as per the number of branches of those vertices. The root has a copy of $HS$ attached. The vertices with $2$ branches are attached with a join $J$, those with a single branch with a cylinder $K$, and the ones with $0$ branches are capped with a half-sphere $HS$. The manifold resulting from attaching the pieces in the prescribed manner is diffeomorphic to $M$, and is denoted by $R_\mathcal{T}$.

$Q_j$ and $R_j$ can consist of multiple components, and we explain the construction with some minor modifications in those cases. On each component $Q_{ji}$ of $Q_j$, the boundaries $\partial^+(Q_{ji})$ and $\partial^-(Q_{ji})$ may each also consist of multiple connected components, and we have $\partial^-(Q_j)=\cup \partial^-(Q_{ji})$ and $\partial^+(Q_j)=\cup \partial^+(Q_{ji})$. The number of components of $R_j$ is always equal to the number of components of $\partial^+(Q_{j})$ by construction. We construct $R_j$ such that $R_j$ is diffeomorphic to $\partial^+Q_j\times [0,1]$ with a disc removed from exactly one component. The pieces $HS$, $K$, or $J$ attached to the vertex represented by $R_j$ can be attached to that component of $R_j$ where the disc has been removed. By the diameter of $\partial^+(P)$ (respectively $\partial^-(P)$) of any connected component $P$, we mean the maximum distance between any two points of $\partial^+(P)$ (resp. $\partial^-(P)$) across all the components of $\partial^+(P)$ (resp. $\partial^-(P)$), where the geodesic connecting the two points lies entirely in $P$. When a piece $P$ has multiple connected components $\{P_i\}$, we consider the maximum of the diameter of $\partial^+(P_i)$ (resp. $\partial^-(P_i)$) across all the components and denote that as $\partial^+(P)$ (resp. $\partial^-(P)$).

For a piece $P$, let $t_P$ and $T_P$ respectively denote the minimum and maximum of the distance function to $\partial^-P$, restricted to $\partial^+P$, across all components. For $k\leq T_P$, let $U_{P,k}$ denote the $k$-tubular neighbourhood of $\partial^-P$, and $v_P(k)=vol(U_{P,k})$, $v_P'(k)=v_P(k)-v_P(k-1)$. Here, a $k-$tubular neighbourhood of any set $S$ in a piece $P$ refers to all points of $P$ that lie within distance $k$ of $S$. When $P$ has multiple components $\{P_i\}$, note that the definition implies that $v_P(k)=\Sigma v_{P_i}(k)$. That is, we consider the volume of the tubular neighbourhoods across all the components of $P$.

Proposition $13$ in \cite{GP} gives us estimates for the volume and diameter bounds for each piece:
\begin{prop}\label{p13}
    Let $Q_j$ be a sequence of possibly disconnected compact manifolds with boundary. Assume that
    \begin{itemize}
        \item $\partial Q_j$ is split into two collections of boundary components $\partial^-Q_j$ and $\partial^+ Q_j$;
        \item $\partial^- Q_{j+1}$ is diffeomorphic to $\partial^+ Q_j.$
    \end{itemize}
    Then there exist integers $l,\ h,\ H$, sequences of integers $t_j,\ u_j,\ U_j,\ d_j$ and Riemannian metrics on pieces $Q_j,\ R_j,\ K,\ HS,\ J$ such that
    \begin{enumerate}
        \item For all components of all pieces $P$, the maximal distance of a point of $P$ to $\partial^- P$ in that component is achieved on $\partial^+ P$ in that component. In other words, the maximum of those distances across the components of $P$ is equal to $T_P$.
        \item $\frac{1}{3}lt_j\leq t_{Q_j} \leq T_{Q_j}\leq lt_j$.
        \item For all other pieces $P$, $\frac{1}{3}l\leq t_{P} \leq T_{P}\leq l$.
        \item diameter$(\partial^- Q_{ji})\leq d_j$ on each component of $Q_{ji}$ of $Q_j$.
        \item All components $P_i$ of all pieces $P$ carry a marked point $y_{P_i}\in \partial^- P$. When a connected component $P_i'$ of $P'$ is glued on top of a component $P_j$ of $P$, $d(y_{P_j},y_{P_{i}'})\leq l$ (resp. $lt_j$ if $P=Q_j$), unless $P=R_j$ and $P'$ is of type $K,\ HS$, or $J$. In that case, $d(y_P,y_{P'})\leq d_j$.
        \item For all pieces $P=K,\ HS,\ J,\ h\leq \textrm{min}v'_{P}\leq \textrm{max}v'_{P}\leq H$.
        \item \textrm{max} $v'_{Q_j}\leq U_j$.
        \item \textrm{max} $v'_{R_j}\leq u_j \leq U_j$.
        \item If $\partial^+Q_j$ and $\partial^-Q_j$ are diffeomorphic, then they are isometric, by an isometry that maps $y_j$ to $y_{j+1}$, and $u_{j+1}=u_j$. $\partial^+ Q_j$ and $\partial^- Q_{j+1}$ are isometric for all $j$.
        \item All pieces have bounded geometry and product metric near the boundary.
        \item $\partial^+ Q_j$ and $\partial^-Q_{j+1}$ are isometric.
         
    \end{enumerate}
    $t_j,\ u_j,\ d_j$ are respectively called the height, volume, and diameter parameters.
\end{prop}

The last point in this proposition, while not explicitly stated in Proposition 13, follows from the proof in \cite{GP}. Some other points have been modified slightly to explicitly account for the situation where $Q_j$ has multiple components.

To show that the volume growth function of the manifold $R_\mathcal{T}$ lies in the same class as the growth function $v$ of the tree, Grimaldi and Pansu define a discrete growth function. First, they define a function $r : R_\mathcal{T} \rightarrow \mathbb{N}$ as follows. If $P = Q_j$ and $x\in P$, $r(x)=\lfloor d(x,\partial^-Q_j)\rfloor+n_jl$. If $P$ is any other type of piece, attached at a vertex of $\mathcal{T}$ of level $n$, and $x \in P$, $r(x) = \lfloor d(x, \partial^-P )\rfloor + nl$. Then the discrete growth function $z$ is defined, for $n \in \mathbb{N}$, as \begin{equation}\label{e1}
    z(n)=\text{vol}\{x\in R_\mathcal{T} |r(x)\leq n\}.
\end{equation}    
    Proposition $17$ of \cite{GP} says that the sequence $\{n_j\}$  can be chosen in such a manner that the discrete growth function $z$ of the resulting $R_\mathcal{T}$ is of the same growth type as $v$. This choice of $n_j$ depends on the sequences in Proposition \ref{p13} and on $v$. Finally, it is shown that the volume growth function of $R_\mathcal{T}$ is of the same growth type as $z$, with a growth constant $3$. Hence, $R_\mathcal{T}$ is the required Riemannian manifold diffeomorphic to $M$ with volume growth in the same class as $v$.


 In Proposition \ref{p13} as stated in \cite{CG}, Grimaldi and Pansu choose a single marked point $y_{Q_i}$ in each $\partial Q_i$. In \cite{GP}, to prove that the volume growth of $R_\mathcal{T}$ lies in the same growth class as $z$, the authors use the points $y_{Q_i}$ to create a path between the basepoint and any other point $x$. This path is then used to estimate the distance of $x$ from the basepoint. Here, the connectedness (path-connectedness) of each $Q_i$ is implicitly assumed. In the case of a manifold with multiple ends, for any exhaustion $\mathcal{A}_i$, there exists some $i$ such that for all $j>i$, $Q_j$ has multiple components. Then this algorithm chooses a marked point $y_{Q_i}$ in only one of those components. Hence for any point $x$ lying in a different component of $Q_i$, this process does not give an estimate for the distance of $x$. Therefore it is necessary to choose different marked points $y_{Q_i}$ in each component of $Q_i$, and hence we have modified the Proposition to account for that scenario.

\textbf{Proof of Proposition \ref{p13}}:

The metrics on the pieces of the form $R_j$, $HS$, $K$, and $J$ require no modification from the construction in \cite{GP}. 
Let us describe the metric on $Q_j$ in detail.

Let $Q_j$ consist of possibly multiple disconnected components. The metric is defined on each component in a similar manner. Consider a component $Q_{ji}$ of $Q_j$. $\partial Q_{ji}$ is divided into the disjoint union of $\partial^+ Q_{ji}$ and $\partial^- Q_{ji}$. Both $\partial^+ Q_{ji}$ and $\partial^- Q_{ji}$ might be disconnected and also consist of multiple components. 

We choose a point $y_{Q_{ji}}$ from any component of $\partial^+ Q_{ji}$.

Put any Riemannian metric of bounded geometry on each of the components of $\partial^+ Q_{ji}$ for each $Q_j$. Since the $\partial^- Q_{ji}$ are diffeomorphic to some boundary component of $\partial^+ Q_{j+1}$, we put a metric on $\partial^- Q_{ji}$ such that this diffeomorphism becomes an isometry. The resulting metric on the $\partial Q_{ji}$ can be extended to some collar neighbourhood of $\partial Q_{ji}$ as a product metric. Then, extend the metric arbitrarily to the rest of the $Q_{ji}$ such that we get a Riemannian metric on each $Q_{ji}$. Completing this procedure for all components of $Q_j$ results in a metric on $Q_j$ such that some tubular neighbourhood of each boundary component has bounded geometry. Call this metric $m_j$. There exists a constant $\lambda_{ji}=\text{max}\{\text{injectivity radius}^{-1},\text{maximum sectional curvature}^{\frac{1}{2}}\}$ such that scaling $m_j$ on $Q_{ji}$ by $\lambda_{ji}$ results in a metric of bounded geometry on $Q_{ji}$, and scaling by $\lambda_j=$max$\{\lambda_{ji}\}$ results in a metric of bounded geometry on entire $Q_j$. We shall use this constant to modify the metric $m_j$ into a warped product metric.

Now, we thicken the boundary on all components of $Q_j$. First, glue a cylinder $\partial Q_j \times [0,T]$ to $\partial Q_j$. This gives a new metric $m_{j,T}$ on $Q_j$, and this metric is a product metric of the form $dt^2+g$ on the $T-$tubular boundary of $Q_j$, where $g$ is the restriction of $m_j$ on $\partial Q_j$. We modify this metric to a warped product metric on this $T-$tubular neighbourhood.

We need to create the warped product metric so that $\partial^+ Q_j$ and $\partial^- Q_{j+1}$ are isometric, while retaining the bounded geometry property of the pieces. Recall that the initial metrics chosen on the boundaries $\partial^+Q_j$ and $\partial^-Q_{j+1}$ were isometric and satisfied the bounded geometry condition, but they were then scaled by different constants in order to ensure that the entire pieces $Q_j$ and $Q_{j+1}$ satisfy the bounded geometry condition. We try to find a warping function such that the boundaries of the $T-$tubular neighbourhoods attached to $\partial^+Q_j$ and $\partial^-{Q_{j+1}}$ have the initial metric on one side and the scaled metric on the other. This ensures that the new boundary components of $\partial^+ Q_j$ and $\partial^- Q_{j+1}$ are isometric by the initial metric chosen on the boundaries. We can then use the metric on these boundary components to construct a cylindrical metric on the piece $R_j$, which is discussed later.
For that, we need to construct a smooth increasing function $f_T:[0,T]\to [1,\infty)$ to get a warped product metric that behaves as we require. The function chosen in \cite{GP} for this purpose does not meet the required criteria, so we choose a different function. Our constraints are:
\begin{itemize}
    \item $f_T(t)=1$ when $0 \leq t \leq 1$.
    \item $f'_T(t)>0$ when $1<t<T-1$.
    \item $f_T(t)=\lambda_j$ when $T-1\leq t\leq T$.
\end{itemize}

We can assume $T=2\alpha$ for some $\alpha$, since $T$ can be chosen as large as we desire. Now, let $\beta$ be the least positive solution to the polynomial equation $x^\alpha(\lambda_j -1)-x^{\alpha - 2}+ \lambda_j=0$. Consider $k=\log{\beta}$.

Define a function $f:[1,T-1]\to[1,\infty)$ as $$f(t)=\frac{(\lambda_j - 1)e^{-\frac{T-2}{x-1}}}{e^{-\frac{T-2}{x-1}}+e^{-\frac{T-2}{T-1-x}}}+1.$$ 
Then, by choosing $T$ to be suitably large, we can ensure that $|\frac{f''}{f}|$ remains bounded by $1$ on the domain of $f$.

Now, define the function $f_T:[0,T]\to [1,\infty)$ in the following manner:
$$
f_T(t)=
\begin{cases}
   
     1 & \text{if }  0\leq t \leq 1;\\
     f(t) & \text{if } 1<t<T-1;\\
    \lambda_j & \text{if } T-1\leq t \leq T.
\end{cases}
$$

Then $f_T$ is a smooth function on its domain, and $f_T$ satisfies the constraints mentioned above.
The metric $m_{j,T}$ can be changed from the product metric $dt^2+g$ to the warped product metric given as $dt^2+f_T(t)^2g$ on the $(T)-$tubular neighbourhood of $\partial Q_j$. On the rest of $Q_j$, scale $m_j$ by $\lambda_j$. The resultant metric is a smooth Riemannian metric at all points of $Q_j$, since at $t=T$, the defined metric is smooth and a product metric in a small enough neighourhood of $T$. Call this new metric $g_{j,T}$. Then $g_{j,T}$ remains a product metric on the $1-$tubular neighbourhood of the boundary $\partial Q_j$ by definition of $f$. If $T$ is large enough, $g_{j,T}$ is a metric of bounded geometry by the boundedness of $|\frac{f''(t)}{f(t)}|$.

We state the following Lemma proven in the paper:
\begin{lem}
    Let $(\Tilde{M},\Tilde{g})$ be a complete and connected Riemannian manifold. Let $M = [-T , 0] \times \Tilde{M}$ be equipped with the warped
product metric $dt^2 + f (t)^2\Tilde{g}$. Let $t_0 \in [-T , 0]$. Let $m, m' \in \Tilde{M}$ . Let $s \mapsto \gamma(s) = (t(s),\Tilde{\gamma}(s))$  denote a minimizing
geodesic from $(t_0, m)$ to $(0, m')$ in $M$. Then 
\begin{enumerate}
    \item  $\Tilde{\gamma}$ is a minimizing geodesic from $m$ to $m'$ in $\Tilde{M}$.
    \item If $d_{\Tilde{g}}(m,m')<e^{-1-t_0} - 3$, then $s \mapsto t(s)$ is monotone with derivative $t'(0) > 0$.
\end{enumerate}
\end{lem}

This Lemma is used to help us choose a value of $T$ large enough such that each component of $Q_j$ satisfies the distance properties mentioned in the Proposition. Here, $T$ varies with the variation of the pieces $Q_j$.
 
Let $D_j$ be the maximum of the diameter of all the components of $Q_j$ with respect to the metric $m_j$. Let that component be $Q_{ji}$. Thus, $D_j=$ diameter$(Q_{ji})$. Then, the diameter of $Q_{ji}$ with respect to the metric $g_{j,T}$ lies between $2T$ and $2T+3D_j$. If we choose $T$ such that $T>2D_j$, then any point of $Q_{ji}$ which lies at maximum distance from the boundary $\partial^-{Q_{ji}}$ lies in the $T-$tubular neighbourhood of $\partial^+{Q_{ji}}$. Let this point be $q$.

In the $T-$tubular neighbourhood of $\partial^+Q_{ji}$, $[-T,0]\times \partial^+Q_{ji}$, let $q$ have the coordinates $(t_0,m)$. Then $T+t_0<D_j$. If $p$ is the point on $\partial^- Q{ji}$ that lies at a minimum distance from $q$, then according to the Lemma mentioned above, all minimizing geodesics from $p$ to $q$ point away from the embedded copies $t\times \partial^+ Q_{ji}$ of $\partial^+ Q_{ji}$ in $[-T,0]\times \partial^+Q_{ji}$. Hence, decreasing the value of $t_0$, that is, pulling $q$ forwards towards $0\times \partial^+ Q_{ji}$ would decrease the distance between $p$ and $q$. Thus, we must have $t_0=-T$, which gives us that with the metric $g_{j,T}$, $q$ lies on the boundary $\partial^+Q_{ji}$.

Every point of $Q_{ji}$ is at a maximum distance of $2T+D_j$ from $\partial^-{Q_{ji}}$ by the construction. The minimum distance of a point of $\partial^+Q_{ji}$ from $\partial^-Q_{ji}$ is at least $T$. Thus, for $T>D_j$, $$T \leq t_{Q_{ji}} \leq T_{Q_{ji}} \leq 2T + 3D_j \leq 3T.$$
For all other components $Q_{jk}$ of $Q_j$, $D_{jk}\leq D_j$. Thus, the same choice of $T$ satisfies the equation $T\leq t_{Q_{jk}}\leq T_{Q_{jk}}\leq 3T$ by the same arguments. Hence, we can choose $T$ for $Q_j$ such that on each component of $Q_j$, the required property is satisfied. That is, on $Q_j$, $T\leq t_{Q_j}\leq T_{Q_j}\leq 3T$. Finally, fix $T$ large enough such that the stated property is satisfied on each component of $Q_j$, as well as $Q_j$ becomes a metric of bounded geometry (i.e, $|\frac{f''}{f}|\leq 1$).

Now we can choose an integer $t_j$ such that $\frac{1}{3}lt_j\leq t_{Q_j}\leq T_{Q_j}+2$diameter$(\partial Q_j)\leq lt_j$. This concludes the determination of the length and distance properties mentioned on points $1, 2, 4, 5, 9, 10, 11$ of Proposition \ref{p13} on the pieces $Q_j$. Now we need to suitably choose a metric on the $R_j$, and determine the volume parameters.

For the metric on $R_j$, recall that $R_j$ is defined as the product of $\partial^+Q_j$ with an interval with a ball removed from one component, we consider product metrics on some tubular neighbourhoods of all the boundary components of $R_j$ such that the components diffeomorphic to $\partial^+ Q_j$ become isometric to it. This metric is now extended to a metric of bounded geometry on the entire $R_j$. Recall that for a piece $P$, $v'_P=v_P(k)-v_P(k-1)$ where the volume $v_P(k)$ is the sum volumes of $P$ on the $k$-tubular neighbourhood of $\partial^-P$ across all the components of $P$. Define $u_j=$max$\{v'_{R_j}\}$. Now we can define $U_j$ as $U_j=$max$\{v'_{Q_j},u_j\}$.

This ensures that all the properties mentioned in the Proposition hold for all pieces.
\qed

\subsection*{Proof of the Theorem \ref{t3}}:

Given any bgd-function $v$, we can choose a bgd-function $w$ of the same growth type as $v$ such that $2 \leq w(n + 2) - w(n + 1) \leq 2(w(n + 1) - w(n))$. Hence, we can assume without loss of generality that the given bgd-function $v$ satisfies $2 \leq v(n + 2) - v(n + 1) \leq 2(v(n + 1) - v(n))$.

\begin{lem}\label{lw}
Let $v : \mathbb{N} \to \mathbb{N}$ satisfy:
\begin{itemize}
\item $v(0) = 1$.
\item For all $n \in \mathbb{N}$, $2 \leq v(n + 2) - v(n + 1) \leq 2(v(n + 1) - v(n))$.
\item $v(n) = O(\lambda^n)$ for some $\lambda < 2$.
\end{itemize}
Fix a subset $S \subset N$ of vanishing lower density, i.e. $\liminf_{n\to \infty}
\frac{|S \cap \{0,...,n\}|}{n} = 0$.
There exists an admissible rooted tree $\mathcal{T}_{S,v}$ with bounded geometry and with growth exactly $v$ at the root.
\end{lem}
Given the function $v$, this Lemma gives us a tree $\mathcal{T}_{S,v}$ subject to the choice of the subset $S$, such that $\mathcal{T}_{S,v}$ has growth exactly $v$ at the root. For any admissible tree $\mathcal{T}$ with growth $v$, we get a Riemannian manifold $R_\mathcal{T}$ by joining the various pieces according to the pattern given by $\mathcal{T}$. We can define a function $r:R_\mathcal{T}\to \mathbb{N}$ where $r$ is defined as the following. If $P = Q_j$ and $x\in P$, $r(x)=\lfloor d(x,\partial^-Q_j)\rfloor+n_jl$. If $P$ is any other type of piece, attached at a vertex of $\mathcal{T}$ of level $n$, and $x \in P$, $r(x) = \lfloor d(x, \partial^-P )\rfloor + nl$. Then the discrete growth function $z$ is defined, for $n \in \mathbb{N}$, as \begin{equation}\label{e1}
    z(n)=\text{vol}\{x\in R_\mathcal{T} |r(x)\leq n\}.
\end{equation}    

The following Lemma lets us choose $S$ as per our requirements. Note that this choice depends only on the sequences $\{U_j\}, \{u_j\}, \{d_j\}$ and $\{ l_j\}$, and not on the number of components of $Q_j$.

\begin{prop}\label{lnj}
    Let $v : \mathbb{N} \to \mathbb{N}$ be a bgd-function that satisfies the assumptions stated in Lemma \ref{lw}.
    
Let $t_j$, $u_j$, $d_j$ be the parameters of the pieces $Q_j$, as provided by the Proposition. Assume that
\begin{itemize}
    \item either $lim_{n \to \infty} v(n + 1) - v(n) = +\infty$;
    \item $u_j$ is constant.
\end{itemize}
Then there exists an increasing sequence $n_j$ such that
\begin{enumerate}
    \item $n_j \leq d_j$;
    \item the subset $S = \bigcup_j [n_j , n_j + t_j - 1]$ has vanishing lower density;
    \item the discrete growth function $z$ of the corresponding Riemannian manifold $R_{\mathcal{T}_{S,v}}$ has the same growth type as $v$.
\end{enumerate}
\end{prop}

Here, the condition $\lim_{n \to \infty} v(n + 1) - v(n) = +\infty$ can be shown to be equivalent to the condition that $\lim_{n\to \infty} \frac{v(n)}{n} = +\infty$. Thus, we have a choice of vertices $S$ of the trunk, and we can choose a manifold $R_{\mathcal{T},v}$ according to the procedure mentioned at the beginning of this section. Using Proposition \ref{p13} we get a Riemannian metric on each piece such that if a piece $P$ is attached on a piece $P'$, the boundary components on which they are attached are isometric and have product metric in some neighbourhood. We attach the pieces in this manner to define a Riemannian metric on $R_{\mathcal{T},v}$ such that the discrete growth function $z$ on $R_{\mathcal{T},v}$ lies in the same growth class as $v$. Note that $R_{\mathcal{T},v}$ is in this case diffeomorphic to the original manifold $M$. We now need to only show that the volume growth function of $R_{\mathcal{T},v}$ also lies in the desired growth class.

To prove the Theorem, it is enough to show that the volume growth function of $R_{\mathcal{T},v}$ lies in the same growth class as $z$, since we already know that $z$ lies in the same growth class as $v$. Choose a base point $o$ on $R_{\mathcal{T},v}$ such that $o$ lies on the piece attached to the root vertex of the tree $\mathcal{T}$. Let $w(n)=$vol$(B(o,n))$ in $R_{\mathcal{T},v}$ with the chosen Riemannian metric. Recall the function $r:R_{\mathcal{T},v}\to \mathbb{N}$. If a point $x$ lies on a piece $P$ of $R_{\mathcal{T},v}$ at level $n$ of $\mathcal{T}$, then $r(x)=\lfloor d(x,\partial^- P)\rfloor + nl$ if $P$ is not of the type $Q_j$ for some $j$, $r(x)=\lfloor d(x,\partial^- P)\rfloor + n_jl$ if $P=Q_j$ and $n_j\leq n < n_j+t_j$.

Choose a shortest geodesic between $o$ and $x$. Such a geodesic passes through some pieces $P_i$ of $R_{\mathcal{T},v}$, where $P_i$ lies on level $i$ of the tree. If the line passes through some piece of the form $Q_j$, i.e if $P_i=Q_j$ for some $i$ then by construction $P_k=Q_j$ for all $n_j\leq k<n_j+t_j$. If $P_i$ consists of multiple components then the line passes through exactly one component of $P_i$ since it is a length minimising geodesic. Let us consider the points $y_{P_i}$ which are the marked points in $\partial^- P_i$ lying in those components through which the geodesic passes. In case $P_i=Q_j$, we get only one marked point for all the $P_i$ where $n_j\leq i < n_j+t_j$. Call those selected points $y_0=o,y_1,y_2,\dots ,y_k\in \partial^- P$. Let $y_{k+1}$ be the point in $\partial^- P$ which is the closest to $o$. Consider the path generated by connecting each $y_i$ to $y_{i+1}$ via minimising geodesics, such that $y_k$ and $y_{k+1}$ are connected via a geodesic lying entirely in $\partial^-P$. Since the distance between $o$ and $x$ is less than the length of this path, we have the bound
$$d(o,x)\leq \Sigma_{i=0}^k d(y_i,y_{i+1})+d(y_{k+1},x).$$
Based on the bounds in the Proposition, we know $d(y_i,y_{i+1})\leq l$ unless $y_i$ is a piece of type $R_j$ for some $j$ and $y_{i+1}$ is a piece of type $HS,$ $K$, or $J$. This is because in these cases, $y_{i+1}$ does not lie on $\partial^+ R_j$. But this might happen only once in the entire path, and in such a situation the revised estimates hold that $d(y_i,y_{i+1})\leq l + $diameter$(\partial^- R_j)\leq l + n_jl$.

If the piece $P$ is of the type $Q_m$ for some $m$, then $\Sigma_{i=0}^{k-1} d(y_i,y_{i+1})\leq 2n_ml$. Since $y_{k+1}$ lies on $\partial^- Q_m$, $d(y_k,y_{k+1})\leq $diameter$(\partial^- Q_m)\leq n_ml$. Thus we get, 
$\Sigma_{i=0}^k d(y_i,y_{i+1})\leq 3n_ml$. Hence,
$$d(o,x)\leq 3n_ml + (r(x)-n_ml)\leq 3r(x).$$
If $P$ is not $Q_m$ for any $m$, $\Sigma_{i=0}^{k-1} d(y_i,y_{i+1})\leq 2nl$, and $d(y_k,y_{k+1})\leq l$, and so 
$$d(o,x)\leq (2n+1)l+(r(x)-nl)\leq \frac{n+1}{n}r(x)\leq 3r(x).$$
For the converse direction, again consider a distance minimizing geodesic $\gamma$ between $o$ and $x$, and let it pass through $n$ pieces (counting each $Q_j$ as $t_j$ pieces). Let $\{s_1, s_2,\dots,s_n\}$ be the values of $s$ such that $\gamma(s)$ lies on $\partial^- P'$ for some piece $P'$. Take $s_0=o$. Then $d(\gamma(s_1),\gamma(s_{i+1}))\geq \frac{l}{3}$, or  $d(\gamma(s_1),\gamma(s_{i+1}))\geq \frac{lt_j}{3}$ if the piece is of the type $Q_j$ for some $j$. And we also have $d(\gamma(s_k),x)\geq d(\partial^-(P),x)$. Therefore we get, if $P$ is of the form $Q_m$ for some $m$,
$d(o,x)\geq \frac{1}{3}n_ml+d(\partial^-(P),x)\geq\frac{1}{3}n_ml+(r(x)-n_ml)\geq \frac{1}{2}r(x)$,
and in other cases, $d(o,x)\geq \frac{1}{3}nl+d(\partial^-(P),x)\geq\frac{1}{3}nl+(r(x)-nl)\geq \frac{1}{3}r(x)$. 

Hence, combining the two directions, 
$$\{x\in R| r(x)\leq \frac{n}{3}\}\leq B(o,n)\leq \{x\in R| r(x)\leq 3n\}.$$
We also have $z(\frac{n}{3})\leq$vol$B(o,n)\leq z(3n)$.

Thus, the volume of $R_{\mathcal{T},v}$ is of the same growth type as $z$, which is of the same growth type as $v$, and so $R_{\mathcal{T},v}$ is our required Riemannian manifold diffeomorphic to $M$ with volume growth in the class of $v$. \qed

\begin{rem}
    If $M$ is an open manifold with countably many compact boundary components, then the same construction as above holds. The metric can be chosen such that it is a product metric in a neighbourhood of each boundary component. We choose an exhaustion such that each $Q_j$ has only finitely many boundary components, and construct a metric on $Q_j$ such that it is a product metric of bounded geometry near each of those boundary components, in a similar manner to the construction above. The rest of the proof follows in the same manner. 
\end{rem}

\section{Volume growth on connected sums along a tree}\label{s3}

Let us consider any finite collection of $m$-dimensional closed manifolds $\mathcal{U}$. Then, given any tree $T$, we can form infinitely many $m$-manifolds by taking connected sums of elements of $\mathcal{U}$ as described in \ref{connectedsum}. We can estimate the volume growth of all such manifolds for any given growth class.

\begin{thm}\label{graph1}
    Let $T$ be any infinite tree without any finite branches, and let $\mathcal{U}$ be a finite collection of $m$-dimensional closed manifolds. Then, given a bgd-function $v$, and a manifold $M$ which is a connected sum of elements of $\mathcal{U}$ along $T$, there exists a Riemannian metric of bounded geometry on $M$ such that the volume growth $v_1$ of $M$ (based on a point chosen on the manifold placed on the root of $T$) lies in the growth class of $v$. The metric can be chosen such that the growth constant of $v$ and $v_1$ is bounded above by a constant depending only on $T$, $v$, $\mathcal{U}$.
\end{thm}

\begin{proof}
    The first statement follows from Theorem \ref{t3}.
    Let the tree $T$ have $F(j)$ vertices on level $j$. Then, $F(j)$ number of disjoint components are attached at the level $j$. We try to rearrange the tree in the form of an admissible tree in the sense of \cite{GP} in order to define a Grimaldi-Pansu metric on the manifold. For that, we take an exhaustion $\{U_i\}$ of $M$ such that $U_{j}$ is the union of all the pieces attached to the vertices until level $j$. Thus, $U_{j}$ consists of $\Sigma_{i=0}^jF(i)$ attached pieces from $\mathcal{U}$. Then, $U_{j+1}\setminus U_j$, which is $Q_j$, is made of $F(j)$ pieces from $\mathcal{U}$ and thus has $F(j)$ number of connected components.

    We first need to put metrics on the $Q_j$ such that each $Q_j$ has bounded geometry, and has product metric in some tubular neighbourhood of the boundaries. Let the cardinality of $\mathcal{U}$ be $\alpha$. By the construction, each connected component of $Q_j$ is an element of $\mathcal{U}$ with some discs removed. Let $Q_{ji}$ be one such component. Then $Q_{ji}$ is one of the $\alpha$ elements from $\mathcal{U}$, and the number of removed discs from the element is equal to the degree of the vertex represented by $Q_{ji}$. Let the degree be $k$. Then $\partial^- Q_{ji}$ consists of a single $(m-1)$-sphere, and $\partial^+ Q_{ji}$ is a union of $k-1$ disjoint $(m-1)$-spheres. For each element of $\mathcal{U}$ with $k$ discs removed, put a metric of bounded geometry such that the metric satisfies all the requirements in the Proposition \ref{p13}. Then on each such piece, the metric satisfies that the distance between any points lying on the positive and negative boundaries is between $\frac{1}{3}lt_{ji}$ and $lt_{ji}$ for some $t_{ji}$. Consider the maximum of the values of $t_{ji}$ across each of the $\alpha$ metrics on the $\alpha$ possible pieces, and thicken the boundaries on all the metrics so that each of them has the distance between the boundaries lying between $\frac{1}{3}lt_{ji}$ and $lt_{ji}$ for that maximum value of $t_{ji}$. Doing this for each component of $Q_j$, and again thickening boundaries if necessary, we get a value $t_j$ such that the distance between the boundaries for any component of $Q_j$ lies between $\frac{1}{3}lt_{j}$ and $lt_{j}$, irrespective of the choice of the element of $\mathcal{U}$ in that component.

    Now, put metrics on the other pieces such that they also satisfy all properties stated in the Proposition \ref{p13}. Consider $U_j$ to be the maximum of $v_{Q_j}'(s)=v_{Q_j}(s)-v_{Q_j}(s-1)$ across all the possible choices for the components of $Q_j$. Similarly, consider $d_j$ to be the maximum of the diameters of $(\partial^-Q_j)$ across all $\alpha$ possible choices for each component of $Q_j$. Then, $\{t_j\}$, $\{U_j\}$, and $\{d_j\}$ are independent on the choice of the element of $\mathcal{U}$ on any vertex of $T$.

    Now, following Proposition \ref{lnj}, there exists a sequence $\{n_j\}$ such that the associated discrete growth function $z$ has the same growth type as $v$. The sequence $\{n_j\}$ depends on the sequences $\{d_j\}$ and $\{t_j\}$, and hence can be chosen to be independent of the choices of the elements of $\mathcal{U}$ at the vertices of $T$. Also, by Lemma $11$ of \cite{GP}, there exists a bgd-function $w$ of the same growth type as $v$ such that $2\leq w(n+2)-w(n+1)\leq 2(w(n+1)-w(n))$. Then, by Lemma $16$ of \cite{GP}, the bounds of the function $z$ are given by
    $$(w(n)-w(n-1)-1)h\leq z(n)-z(n-1)\leq H(w(n)-w(n-1))+U_j$$
    where $h$ and $H$ are respectively the mimimum and maximum values of $v'_P$ across the pieces $P=HS, K, J$, as defined in the Proposition \ref{p13}. Thus, the bounds for $z$ are independent of the choice of the elements representing the vertices of $T$, and depend only on the set $\mathcal{U}$, the function $w$, and the tree $T$. Therefore, the growth constant of $w$ and $z$ is also independent of the choice of the representative elements of the vertices. Since the growth constant of $z$ and the volume function $v_1$ of $M$ can be chosen to be the constant $3$, and the growth constant between $w$ and $v$ depends only on $v$, we get that the growth constant between $v_1$ and $v$ is hence independent of the choice of the elements at the vertices of $T$ and depends only on $T$, $\mathcal{U}$, and $v$. This completes the proof.
\end{proof}

In the case where the degree of any vertex of a tree $T$ is bounded above by some constant $k$, we can consider $T$ as a subtree of the tree with degree $k$ at each vertex. Hence in such situations, the growth constant for any given bgd-function can be made independent of the exact choice of the tree, and depends only on the collection of the manifolds we attach to the vertices and the bound $k$. 

\begin{thm}\label{graph2}
    Let $T$ be an infinte rooted tree with no finite branches such that the degree of any vertex of $T$ is bounded by some integer $k$. Let $\mathcal{U}$ be a finite collection of $m$-dimensional closed and compact manifolds. Then, given a bgd-function $v$, and any manifold $M$ which is a connected sum of elements of $\mathcal{U}$ along $T$, there exists a Riemannian metric of bounded geometry on $M$ such that the volume growth $v_1$ of $M$ (based on a point on the manifold attached to the root of $T$) lies in the growth class of $M$. The metric can be chosen such that the growth constant of $v$ and $v_1$ is bounded above by a constant depending only on $v$, $k$, and $\mathcal{U}$.
\end{thm}

\begin{proof}
    Consider a rooted tree $T$ such that each vertex of $T$ has degree at most $k$, and let us attach elements of $\mathcal{U}$ at each vertex of $T$ to get the manifold $M$. We can take an exhaustion $\{U_i\}$ of $M$ such that $U_{j}$ is the union of all the pieces attached to the vertices until level $j$. Then $Q_j$=$U_{j}\setminus U_{j-1}$ is a disjoint union of all the pieces attached to vertices at level $j$.

    Similar to the previous theorems, we try to put metrics of bounded geometry on each element of $\mathcal{U}$ such that they satisfy the conditions in the Proposition \ref{p13}. Let the cardinality of $\mathcal{U}$ again be $\alpha$. Consider the sets $\mathcal{U}_i$ consisting of the elements of $\mathcal{U}$ with $i$ discs removed, where $0\leq i\leq k+1$. So, each element of $\mathcal{U}_i$ has $i$ boundary $(m-1)$-spheres, and $\mathcal{U}_0=\mathcal{U}$. Put a metric of bounded geometry on each element of $\mathcal{U}_i$ for all the $i$. Then we get $\alpha^{k+2}$ different pieces with Riemannian metrics of bounded geometry on them such that they are product metrics on some tubular neighbourhoods of each boundary component. Now, each $Q_j$ consists of a disjoint union of at most $k^j$ of these pieces, since the number of vertices on level $j$ is at most $k^j$. Again, thicken the boundary on each of these pieces so that the value of $t_j$ becomes constant for each $Q_j$, i.e, so that every component of each $Q_j$ satisfies $\frac{lt_0}{3}\leq t_{Q_j}\leq T_{Q_j}\leq lt_0$ for some integer $t_0$. Consider $d_0=$  $\max\{\text{diameter}(\partial^-P)\}$, where $P$ is one of the  $\alpha^{k+2}$ pieces which are elements of some $\mathcal{U}_i$. Define $\{d_j\}$ to be the constant sequence $\{d_0\}$, so that diameter$(\partial^-Q_j)\leq d_j$ for all $j$. Also, let $U$ be the maximum value of $v'_P$ across all the pieces $P$ in all $\mathcal{U}_i$, and let $u$ be the minimum value of $v'_P$. Note that $U$ and $u$ depend on the collection $\mathcal{U}$ and the maximum degree $k$. Then, since $Q_j$ has a maximum of $k^j$ disjoint components, $v'_{Q_j}\leq Uk^j$. For the other pieces of the form $R_j,$ $HS,$ $J$, and $K$, we put metrics of bounded geometry on them satisfying the conditions stated in the Proposition \ref{p13}. Let $h$ and $H$ be the minimum and maximum values of $v'_P$ where $P$ is any piece of the form $HS$, $K$, or $J$.
    
    By the above construction, we again get sequences $\{U_j\}$, $\{t_j\}$, $\{d_j\}$ independent of the choice of representative elements of the vertices, and also independent of the tree $T$. Thus, applying Proposition \ref{lnj}, there exists a sequence $\{n_j\}$ independent of the tree $T$ and the choice of elements at the vertices of $T$ such that the discrete growth function $z$ resulting after attaching the pieces to an admissible tree according to the sequence $\{n_j\}$ is of the same growth class as $v$. Thus, the same choice of the set $S$ works for all trees with the number of branches at any vertex bounded above by $k$. 

    Again using Lemma $11$ of \cite{GP}, we get a bgd-function $w$ of the same growth type as $v$, such that $2\leq w(n+2)-w(n+1)\leq 2(w(n+1)-w(n))$.
    Then, the function $z$ satisfies the following bounds:
    $$(w(n)-w(n-1)-1)h\leq z(n)-z(n-1)\leq H(w(n)-w(n-1))+U_j$$
    where $ln_j\leq n < ln_{j+1}$.
    $\{n_j\}$ can be chosen such that for all $ln_j\leq n \leq ln_{j+1}$, $w(n)-w(n-1)\geq U_j = k^jU$.
    Taking any constant $\beta<{u,h}$ and using that $w(n)-w(n-1)\geq U_j = k^jU$ for $ln_j\leq n \leq ln_{j+1}$ and that $w$ is an increasing function, the bounds can be modified to 
    $$(w(n)-w(n-1))\beta \leq z(n)-z(n-1)\leq (H+1)(w(n)-w(n-1))$$
    for all $n$.
    Substituting the values $w(0)=1$ and $z(0)=0$ and iterating, we get $$z(n)\leq (H+1)w(n),$$
    and $$(w(n)-1)\beta \leq z(n)$$
    which implies that
    $$w(n)\leq \frac{z(n)}{\beta}+1.$$
    Then, choosing $A=$ max$\{\frac{1}{\beta},H+1\}=$ max$\{\frac{1}{u},\frac{1}{h},H+1\}$, we get that $w$ and $z$ are in the same growth class with growth constant $A$. Also, by \cite{GP}, $z$ and the volume growth function of $M$ are in the same growth class with growth constant $3$, and $w$ and $v$ are in the same growth class with growth constant $L^{\frac{1}{l}}$, where $L$ satisfies $\frac{1}{L}\leq v(n+2)-v(n+1)\leq L(v(n+1)-v(n))$.  Then, the volume growth function of $M$ and $v$ lie in the same growth class, with growth constant $3L^{\frac{1}{L}}\times\max\{\frac{1}{u},\frac{1}{h},H+1\}$. Since $L, u, h,$ or $H$ do not depend on the choice of the tree $T$, the growth constant is independent of $T$ and depends only on $\mathcal{U}$, $k$, and $v$.

    Thus, on any such manifold $M$ formed from a tree with the degree of any vertex bounded above by $k$, we can put a Grimaldi-Pansu metric with the resulting volume growth function $v_1$ having the bounds $v_1(n)\leq Kv(Kn+K)+K$ with $K=3L^{\frac{1}{L}}\times\max\{\frac{1}{u},\frac{1}{h},H+1\}$.
\end{proof}
\subsection*{ Proof of Theorem \ref{graph} :}
    We again choose an exhaustion $U_j$ of $M$ such that $U_j$ consists of the pieces attached to the vertices until level $j$, and take $Q_j=U_j\setminus U_{j-1}$. If $T$ does not have any finite branch then the theorem follows from Theorem \ref{graph1} and Theorem \ref{graph2}. If $T$ has finite branches, then there exist components of some $Q_j$ such that there is no piece $Q_{j+1}$ or $R_j$ attached to that component. Consider such a component $Q_{ji}$. Then $\partial^- Q_{ji}$ is empty. To put a metric on $Q_{ji}$, put a metric of bounded geometry on $\partial^+ Q_{ji}$ by making it isometric to the component of $\partial^- Q_{j-1}$ to which it is diffeomorphic. Extend this metric to a metric of bounded geometry on the entire $Q_{ji}$ so that it is a product metric on some tubular neighbourhood of the boundary. Proceeding in a similar manner as the proof of Theorem \ref{t3}, we can get a Grimaldi-Pansu metric on $M$ such that the volume growth function lies in the growth class of $v$. Then, following a similar proof as in Theorem \ref{graph1} and Theorem \ref{graph2} the claim follows. 
\hfill $\square$

\begin{rem} Let $M$ be a connected open manifold which is a connected sum of manifolds from a finite set of closed $3$-manifolds $\mathcal{U}$ along a simple locally finite graph. Following a construction similar to Theorem $2.3$ in \cite{BBM}, we can replace the graph with some tree $T$ such that the manifold obtained by taking connected sum along $T$ is diffeomorphic to that obtained by taking connected sum along the graph. We can then use the Theorem \ref{graph} to get a bound for the growth constant, which shall depend only on the graph and the set $\mathcal{U}$ of closed manifolds.
\end{rem}

\begin{cor}
Let $M$ be a manifold as described in Theorem \ref{graph2}. For any bdg-function $v$ and a point $x\in M$ there exists a Grimaldi-Pansu metric $g_x$ on $M$ such that the volume growth function $v_x$ based at $x$ belongs to the class of $v$. The growth constant of $v$ and $v_x$ is bounded by a constant depending only on $v$, $\mathcal{U}$ and $k.$ 
\end{cor}
\begin{proof} By the definition of $M$ any point $x\in M$ belongs to an element of $\mathcal{U}$, which is placed at a vertex $\alpha$ of the tree $T$. Choose $\alpha$ to be the root of $T$. Then the corollary follows from Theorem \ref{graph2}. 
\end{proof}
The Grimaldi-Pansu metric constructed in the proof of Theorem \ref{t3} is heavily dependent on the choice of base point. Theorem \ref{t3} only ensures the existence of a volume growth function in the same class as a given function but does not say anything about the growth constant between them.  While the volume growth functions based at different basepoints all lie in the same growth class, the growth constants might differ largely, and that stops us from having any proper bounds on the volume functions. However, as a consequence of the above corollary, the growth constants of $v$ and the volume growth functions $v_x$ of the family of metrics $g_x$ are uniformly bounded by a constant. Moreover, if $v$ is doubling then we have the following result.
\begin{cor} Let $M$ be a manifold as described in Theorem \ref{graph2}. For any bdg-function $v$ which satisfies a doubling condition as defined in \ref{doublingdfn} there exist positive constants $r_0,A$ and Grimaldi-Pansu metrics $g_x$ on $M$ such that
$$\frac{vol_{g_x}(B(x,2r))}{vol_{g_x}(B(x,r))}\leq A, \quad \forall r>r_0 \ \ {\rm and} \  \ \forall x\in M .$$
\end{cor}
\begin{proof} From the above corollary we obtain that the growth constants of the volume growth functions of $g_x$ and $v$ are uniformly bounded. Hence there exist positive constants $A_1, A_2,n_0$ such that
\be A_1 v(A_1n)\leq Vol_{g_x}(B(x,n))\leq A_2v(A_2n), \quad \forall x\in M, \ \forall n\geq n_0.\ee
Since $v$ satisfies the volume doubling condition there exists a constant $K>0$ such that $ \frac{v(2n)}{v(n)} \leq K$ for all $n\in \mathbf{N}$. Combining these two inequalities the result follows.
\end{proof}

\section{Some geometric properties}\label{s4}

A. Grigor'yan and L. Sallof-Coste introduced the Relatively Connected Annulus (R.C.A.) condition on manifolds with one end, which is defined as follows.

\begin{defn}
    A Riemannian manifold $M$ is said to satisfy the R.C.A. condition with respect to a reference point $o\in M$ if there exists a positive constant $0<\theta<1$ such that for any $r>\frac{1}{\theta^2}$ and all $x,y\in M$ with $d(o,x)=d(o,y)=r$, there exists a continuous path from $x$ to $y$ staying in $B(o,\frac{r}{\theta})\setminus B(o,r\theta)$.
\end{defn}

In \cite{CG}, G. Carron considered the equivalent definition along with the added constraint that the length of the geodesic path $\gamma$ joining $x$ and $y$ should satisfy $\mathcal{L}(\gamma)\leq \frac{r}{\theta}$. In the case of manifolds with finitely many ends, he also introduced the Relatively Connected to an End (R.C.E.) condition, which is an extension of the R.C.A. condition. The two conditions are equivalent when the manifold has just a single end.

\begin{defn}
    A complete Riemannian manifold $M$ with a finite number of ends is said to satisfy the R.C.E. condition if there is a constant $0<\theta<1$ and some $r_0$ such that for all $r>r_0$ and for any point $x\in \partial(B(o,r))$, there exists a continuous path $c:[0,1]\to M$ such that $c(0)=x$, $c[0,1]\subset B(o,\frac{r}{\theta})\setminus B(o,r\theta)$, $\mathcal{L}(c)\leq \frac{r}{\theta}$, and there exists a geodesic ray $\gamma$ lying in $M\setminus B(o,\frac{r}{\theta})$ such that $\gamma(0)=c(1)$.
\end{defn}

In what follows, we show that the metrics defined on the manifolds discussed in the previous section can be modified so that they satisfy the R.C.A. or R.C.E. condition depending on their number of ends, while their volume growth remains in a given class $[v]$. The key point here is to select the sequence $\{n_j\}$ in such a manner that we can ensure that the finite branches of the underlying tree with growth in class $[v]$ terminate as we desire. In the case of manifolds with a single end, we have the following theorem. 

\begin{thm}\label{rca1}
     Let $\mathcal{U}$ be a finite collection of closed manifolds of dimension $m$. If $M$ is a one-ended manifold such that $M$ is a connected sum of elements of $\mathcal{U}$ along the rooted tree having one end and no finite branches, then given a bgd-function $v$, there exists a Grimaldi-Pansu metric on $M$ such that the volume growth function on $M$ lies in the same growth class as $v$, and $M$ satisfies the R.C.A. condition.
\end{thm}

\begin{proof}
    By Lemma $11$ of \cite{GP}, we can assume without loss of any generality that $v(n)=o(\lambda^n)$ for some $1<\lambda<2$.
    
    Following the previous theorems, we can choose an exhaustion on $M$ such that the pieces of the form $Q_j$ are elements of $\mathcal{U}$, and we can put metrics on each such piece so that they satisfy $\frac{1}{3}lt_0\leq t_{Q_j}\leq T_{Q_j}\leq lt_0$ for some $t_0$, as per the requirements in the Proposition \ref{p13}. Fix a basepoint $o$ on the piece $M_1$. We claim that the set $S=\{n_j\}$ of vertices where we attach the pieces $Q_j$ can be chosen in such a manner that the resultant manifold satisfies the R.C.A. condition. Let $\mathcal{T}$ be the admissible tree with growth $v$ where we attach the pieces.

    Consider a point $x$ lying on $\partial B(o,r)$. Then $x$ lies on a piece of level $n$ where $\frac{r}{l}\leq n\leq \frac{3r}{l}$, by Proposition \ref{p13}. For any $0<\theta<1$, any point lying at a distance $r\theta$ from $o$ lies on a point at some level $n_\theta$ where $\frac{r\theta}{l}\leq n_\theta\leq \frac{3r\theta}{l}$. Therefore, the annulus $B(o,r)\setminus B(o,r\theta)$ has a minimum of $\lfloor \frac{r(1-\theta)}{l}\rfloor$ levels in the tree $\mathcal{T}$. Then, it suffices to show that any vertex on level $s=\lfloor \frac{r}{l} \rfloor$ is connected to the trunk within the annulus, so it can be connected to the trunk after level $\lceil \frac{3r\theta}{l}\rceil$. Thus, it suffices to show that a vertex on level $s$ departs the trunk only after level $3\theta s +1$. However, in the case the vertex on the trunk at level $\lfloor 3\theta s +1\rfloor$ is a $Q_j$, we need to show that a vertex on level $s$ departs the trunk only after level $3\theta s +1+t_0$. Note that for $r$ large enough and $\theta$ small enough, if a piece intersects $\partial (B(o,r))$ then the entire piece lies within the annulus $B(o,\frac{r}{\theta})\setminus B(o,r\theta)$.

    First, let us assume that $\lim_{n\to \infty}v(n)-v(n-1)\to \infty$. Then $\{n_j\}$ can be chosen such that $n_j>j!K$ for a large enough constant $K$, by Proposition  \ref{lnj}. This ensures that for any $k>1$ and any annulus containing the vertices in levels $[n,kn]$, we can choose $n_j$ such that there is at most one $Q_j$ present in the annulus for all $n$.  

    There are $s-(3\theta s+1+t_0)$ levels of vertices in the annulus, and we choose the sequence $\{n_j\}$ such that the interval $[3\theta s +1+t_0,m]$ contains at most one value of $n_j$. Then, since each piece $Q_j$ is attached to $t_0$ vertices, there are at least $s-(3\theta s+1)-2t_0$ many levels of vertices in the annulus which do not have a $Q_j$ attached. Each of these vertices can have $2$ branches attached at each level. Since at any level $k$, we attach $v(k)-v(k-1)$ branches starting from the newest branch of the trunk, the only way that there exists a vertex at level $s$ that departs before level $3\theta s +1+t_0$ is if each of the vertices in the $s-(3\theta s+1)-2t_0$ levels in the annulus have $2$ branches at each subsequent level. But then, there are at least $\Sigma_{i=1}^{s-(3\theta s+1)-2t_0} 2^i$ vertices at level $s$, since a vertex of the trunk at level $3\theta s+1+t_0+i$ which is not represented by $Q_j$ branches to $2^{s-((3\theta s+1+t_0)+i)}$ vertices by level $s$. Therefore,
    $$v(s)> \Sigma_{i=1}^{s-(3\theta s+1)-2t_0} 2^i>2(2^{s-(3\theta s+1)-2t_0})=2^{-2t_0}(2^{(1-3\theta)})^s.$$
    We know that $v(n)=o(\lambda^n)$ for some $\lambda<2$. Thus, there exists $N_1>0$ such that $v(n)\leq \alpha \lambda^n$ for some constant $\alpha$. 
    
    Choose $\theta< \min\{{\frac{1}{3}(1-\frac{\ln{\lambda}}{\ln{2}}), \frac{-3\ln{2}+\sqrt{9(\ln{2})^2-4\ln{\frac{\lambda}{2}}\ln{\alpha 2^{2t_0}}}}{2\ln{\alpha 2^{2t_0}}} }\}$. This satisfies the constraint that $0<\theta<\frac{1}{3}$ since $\lambda<2$. Then we have $$2^{(1-3\theta)}>\lambda.$$
    Then, for all $n>\max\{N_1, \frac{\ln{\alpha}+ \ln{2^{2t_0}}}{(1-3\theta)\ln{2}-\ln{\lambda}}\}$, $${(\alpha 2^{t_0})}^{\frac{1}{n}}\lambda<2^{(1-3\theta)}.$$
    This implies that $$2^{-2t_0}(2^{(1-3\theta)})^s>\alpha \lambda^n.$$
    If the manifold does not satisfy R.C.A. at level $s$ then we have shown that $v(s)> 2^{-2t_0}(2^{(1-3\theta)})^s.$ Thus, using the fact that $v(n)\leq \alpha \lambda^n$, we arrive at a contradiction.

    By our choice of $\theta$, whenever $r>\frac{1}{\theta^2}$, that is $n>\frac{1}{\theta^2l}$, the vertices at level $n$ depart from the trunk at some level after $3\theta n + 1+t_0$. Thus we can ensure that any point at distance $r$ from $o$ reaches the trunk at some point at a distance greater than $r\theta$ from $o$. This implies that any $2$ points lying on $\partial B(o,r)$ can be connected via a path lying in the annulus $B(o,\frac{r}{\theta})\setminus B(o,r\theta)$ for all $r>\frac{1}{\theta^2}$, and hence the manifold satisfies the R.C.A. condition. 

    Now, consider the case where $v(n)-v(n-1)$ is bounded above by some constant $A$. Then $v(n)\leq An+v(1)$. In this case $\{n_j\}$ can be chosen to satisfy $n_j=C_1+jC_2$ for some constants $C_1$ and $C_2$, also by Proposition \ref{lnj}. We again try to show the existence of some $\theta$ such that any piece on level $s$ departs from the trunk after level $s-(3\theta s + 1+t_0)$. Let there be $\mu$ number of pieces of the type $Q_j$ in the annulus. Then there are at least $\mu-1$ intervals of $C_2$ vertices each where the pieces attached to the trunk are not of the type $Q_j$. Therefore, $$C_2(\mu-1)+t_0\mu=s-(3s\theta +1 +t_0).$$
    This implies $$\mu=\frac{s(1-3\theta)+1}{t_0+C_2}+1.$$
    From each of the $(\mu-1)C_2$ vertices, there are $2$ vertices from the trunk. Like in the previous case, each of those vertices can have $2$ branches at each subsequent level, and since these need to be exhausted at level $s$ before a branch from beyond the annulus reaches level $s$, we need to have $$v(s)> \Sigma_{i=1}^{C_2(\mu-1)}2^i=\Sigma_{i=1}^{\frac{C_2(s(1-3\theta)+1)}{t_0+C_2}} 2^i>2(2^{\frac{C_2(s(1-3\theta)+1)}{t_0+C_2}})=2^{(\frac{C_2}{t_0+C_2}+1)}(2^{\frac{C_2(1-3\theta)}{t_0+C_2}})^s.$$
    Since we have some $\lambda<2$ such that $v(n)=o(\lambda^n)$, which implies that $v(n)\leq \alpha \lambda^n$ for some constant $\alpha$, we can again choose $C_2$ large enough and $\theta$ small enough such that there exists some $N$ so that $2^{(\frac{C_2}{t_0+C_2}+1)}(2^{\frac{C_2(1-3\theta)}{t_0+C_2}})^s>\alpha\lambda^s$ for all $s>N$, and $N<\frac{1}{l\theta^2}$. Thus we again arrive at a contradiction.
    
    Hence for any piece lying at a level $s>N$, there is a geodesic path joining it to the trunk lying within the pieces between levels $3s\theta +1+t_0$ and $s$. Therefore, for all $r$ large enough ($r>Nl$), any point on $\partial (B(o,r))$ can be connected to a point on the trunk via a path lying in the annulus $B(o,\frac{r}{\theta})\setminus B(o,r\theta)$.
    
    Thus, the manifold satisfies the R.C.A. condition as claimed.
\end{proof}

\begin{thm}\label{rca2}
    Given a one-ended manifold $M$ and a bgd-function $v$ such that  $\lim_{n\to \infty}v(n)-v(n-1)\to \infty$, there exists a Grimaldi-Pansu metric on $M$ such the volume growth function lies in the growth class of $v$ and $M$ satisfies the R.C.A. condition.
\end{thm}

\begin{proof}
    Since $M$ has a single end, we can take an exhaustion $\{U_i\}$ of $M$ such that each $Q_j=U_{j+1}\setminus U_j$ is connected. Next, we put metrics of bounded geometry on each $Q_j$ in a manner that satisfies the constraints of Proposition \ref{p13}. So we get a sequence $\{t_j\}$ such that the piece $Q_j$ is attached to $t_j$ number of consecutive vertices. Also, the diameter of each $Q_j$ is bounded above by $lt_j$.

    By Lemma $11$ of \cite{GP}, we again assume without loss of any generality that $v(n)=o(\lambda^n)$ for some $1<\lambda<2$. Thus there exists some integer $N$ such that for all $n>N$, $v(n)\leq \alpha\lambda^n$ for a constant $\alpha$. Fix some $\theta< \frac{1}{3}(1-\frac{\ln{\lambda}}{\ln{2}})$.

    Like in the previous proof, it suffices to show that there exists some $N_1$ such that for any $s>N_1$, any piece lying on a level $s$ departs from the trunk after level $3\theta s+1$. We claim that we can choose the sequence $\{n_j\}$ of the levels where the pieces $Q_j$ are attached in such a way that this holds.

    If there is a piece $Q_j$ attached on some vertices within the annulus between levels $3\theta s+1$ and $s$, then there are at most $s-(3\theta s +1+t_j)$ levels of the trunk which are not of type $Q_j$ in this annulus. Following the method used in the previous proof, if we consider $s>\frac{\ln{\alpha}+\ln{2^{t_j}}}{(1-3\theta)\ln{2}-\ln{\lambda}}$, then the piece on level $s$ departs from the trunk within the annulus. Thus, we can choose $n_j>\frac{\ln{\alpha}+\ln{2^{t_j}}}{(1-3\theta)\ln{2}-\ln{\lambda}}$. Note that if there is a piece $Q_j$ attached at level $s$, then the choice of $n_j$ ensures that $s-(3\theta s+1)>t_j$. 

    In the case where there is no piece of the type $Q_j$ attached to any vertex in the annulus, we can follow the same proof to see that for all $s>\frac{\ln{\alpha}}{(1-3\theta)\ln{2}-\ln{\lambda}}$, the vertices on level $s$ can be joined to the trunk within the annulus.

    Thus, choosing $\theta< \min\{\frac{1}{3}(1-\frac{\ln{\lambda}}{\ln{2}}),\frac{-3\ln{2}+\sqrt{9(\ln{2})^2-4\ln{\frac{\lambda}{2}}\ln{\alpha}}}{2\ln{\alpha }} \}$, and $n_j>\frac{\ln{\alpha}+\ln{2^{t_j}}}{(1-3\theta)\ln{2}-\ln{\lambda}}$, we get that the manifold satisfies the R.C.A. condition for all levels $s>\frac{\ln{\alpha}}{(1-3\theta)\ln{2}-\ln{\lambda}}$, that is for all $r>\frac{1}{\theta^2}$.
\end{proof}

\begin{rem}
    In Theorem \ref{rca1}, the components of the pieces $Q_j$ belonged to finitely many diffeomorphism classes, and hence we were able to obtain bounds on the diameter of the $Q_j$. This ensured that we were able to choose the sequence $\{n_j\}$ in such a manner as to ensure the R.C.A. condition even when $v(n)-v(n-1)$ was bounded. However, in the general case, since we cannot obtain diameter bounds for the $Q_j$, we require the added constraint. When $v(n)-v(n-1)$ is bounded, the sequence $\{n_j\}$ needs to be linear by Proposition \ref{lnj}. Hence we cannot choose $n_j$ to be arbitrarily large, and the construction above does not work.
\end{rem}

Note that the geodesic path joining any two points of $\partial B(o,r)$ constructed in the above theorems also satisfies Gilles Carron's additional constraint on length, since the maximum length of any such path is $2(2l + r(1-\theta))<\frac{r}{\theta}$.
When the manifold has multiple ends, the same idea can be extended to prove the R.C.E. condition. 

\begin{cor}\label{rce1}
    Given a finite collection $\mathcal{U}$ of closed manifolds of the same dimension, any infinite rooted tree $T$ with finitely many ends and no finite branches, and any bgd-function $v$, for every manifold $M$ which is a connected sum of elements of $\mathcal{U}$ along $T$ there exists a Grimaldi-Pansu metric on $M$ such that the volume growth function of the manifold lies in the same growth class as $v$, and $M$ satisfies the R.C.E. condition.
\end{cor}

\begin{proof}
    We can again consider a suitable exhaustion of $M$ such that each $Q_j$ is a disjoint union of elements of $\mathcal{U}$ with discs removed, and put suitable metrics on all pieces as in the previous proofs. On choosing the sequence $\{n_j\}$ of vertices of the trunk where we attach a piece $Q_j$, we get a tree $\mathcal{T}$ with growth $v$. The pieces are attached according to the vertices of the tree $\mathcal{T}$ as before. Note that all the pieces attached to the vertices of the tree $T$ are now attached to the trunk of the tree $\mathcal{T}$.
    On any piece $P$ attached to the trunk of $\mathcal{T}$ such that $\partial(B(o,r))$ intersects $P$, there exists a point $x\in P\cap \partial(B(o,r))$ and a geodesic ray $\gamma:[1,\infty)\to M\setminus B(o,r)$ such that $\gamma(0)=x$. Thus, it again suffices to show that there exists a $\theta<1$ such that for any point $y$ in $M$ with $y\in \partial(B(o,r))$, $y$ can be joined to a piece on the trunk via a geodesic path lying entirely in the annulus $B(o,\frac{r}{\theta})\setminus B(o,r\theta)$. Therefore the proof of the previous theorem also works in this case and suffices to prove that the manifold satisfies the R.C.E. condition.
\end{proof}

\begin{cor}\label{rce2}
    Given a manifold $M$ with finitely many ends and a bgd-function $v$ which satisfies  $\lim_{n\to \infty}v(n)-v(n-1)\to \infty$, there exists a Grimaldi-Pansu metric on $M$ such the volume growth function lies in the growth class of $v$ and $M$ satisfies the R.C.E. condition.
\end{cor}

\begin{proof}
    Since $M$ has finitely many ends, we can take an exhaustion $\{U_i\}$ such that $M\setminus U_1$ consists of finitely many connected components of one end each. Then setting $Q_j=U_{j+1}\setminus U_j$, we can follow the same construction as in the case of a manifold with a single end to show that $\{n_j\}$ can be chosen in such a manner that $M$ satisfies the R.C.E. condition.
\end{proof}

Similar results can be obtained in the case of manifolds with infinitely many ends, provided that the exhaustion $\{U_j\}$ is chosen carefully to ensure that no connected component of $M\setminus U_j$ is bounded. We now turn to studying some other properties of the Grimaldi-Pansu metrics.

The Grimaldi-Pansu metric depends on the choice of base point of the manifold. While the growth classes of the volume growth functions based at different base points are always the same, the growth constants differ. In what follows, we study the restrictions on the growth class of the volume growth functions under the assumption of certain basepoint-free restrictions.
\begin{defn}\label{doublingdfn}
    A Riemannian manifold $(M,g)$ is said to satisfy a volume doubling condition if there exists a constant $K>0$ such that 
    $$\frac{vol(B(x,2r))}{vol(B(x,r))}\leq K, \quad \forall x\in M, \ \forall r>0.$$
 A bgd-function $v$ is said to satisfy a doubling condition if there exists a constant $K>0$ such that 
 $\frac{v(2n)}{v(n)}\leq K $ for all positive integer $n.$
\end{defn}
 Any polynomial satisfies the doubling condition. By \cite{CG}, if the volume growth function of a manifold satisfies the doubling condition, then the manifold must necessarily have finitely many ends. If $M$ is a manifold with infinitely many ends then for any polynomial $p$ there exists a Grimaldi-Pansu metric on $M$ such that the growth of volume lies in the class of $p$ but that metric does not satisfy the doubling condition. Therefore, its volume growth function is not independent of the choice of the base point. 
\begin{thm}\label{doubling}
    For a manifold $M$ and a bgd-function $v$ that satisfies the doubling condition, if there exists any Grimaldi-Pansu metric on $M$ such that for any constants $r_0$, $A$, $B$, and $\alpha\geq 0$ the volume growth function on $M$ satisfies
    \be\label{4.1} Av(r)\leq vol(x,r)\leq Bv(r)r^\alpha,\quad \forall x\in M \ \ {\rm and} \ \forall r>r_0,\ee then there exists a constant $C>0$ such that $v(n)\leq Cn^{2\alpha + 1}+C$.
\end{thm}

\begin{proof}
    For any bgd-function $v$, the construction of the admissible tree $\mathcal{T}$ ensures that the number of components at a level $n$ is at least $v(n)-v(n-1)$. 
     
    Consider a basepoint $o$ lying on the piece attached to the root vertex of the tree $\mathcal{T}$. For any $r$, we wish to estimate the number of connected components in the annulus $A=B(o,3r)\setminus B(o,r)$. Consider the ball $B(o,3r)$  of radius $3r$ centered at $o$, and from each connected component of the annulus, take a single point lying on $\partial(B(o,2r)) $. Let $\{x_\alpha\}$ be the collection of those points. Then each $B(x_\alpha , r)\subset A$, and all such balls are disjoint from each other. Also, $\bigcup_\alpha B(x_\alpha, r))\subset B(x_{\alpha},4r)$ for all $\alpha$. We can now choose $r$ such that the annulus contains the components at the $n^{th}$ level of the tree. Choosing $r=\frac{nl}{3}$ is sufficient, since the components at level $n$ lie at a minimum distance of $\frac{nl}{3}$ from the basepoint, and at a maximum distance of $nl$, by the constraints in Proposition \ref{p13}. Therefore, the annulus $B(o,nl)\setminus B(o,\frac{nl}{3})$ has at least $v(n)-v(n-1)$ components.

    Let $v$ satisfy the doubling condition. If the volume growth function on $M$ satisfies $Av(r)\leq vol(x,r)\leq Bv(r)\alpha^r$ for all $x$ and for all $r>r_0$, we have $\frac{vol(x,2r)}{vol(x,r)}\leq \frac{Bv(r)r^{\alpha}}{Av(r)}\leq Kr^\alpha$ for some constant $K$. 
    Then, for any $x_\beta \in \{x_\alpha\}$, $$ \Sigma_\alpha vol(B(x_\alpha,r))\leq vol(B(x_\beta,4r)) \leq 4K^2r^{2\alpha}vol(B(x_\beta,r)).$$
    Hence there are at most $4K^2(nl)^{2\alpha}$ connected components in the annulus $B(o,nl)\setminus B(o,\frac{nl}{3})$. Therefore, $v(n)-v(n-1)\leq 4K^2(nl)^{2\alpha}$.

    Therefore, there exists some constant $C$ such that $v(n)\leq Cn^{2\alpha +1}+C$.
\end{proof}

 If the volume growth function of a Riemannian manifold satisfies a doubling condition then it necessarily has finitely many ends by \cite{CG}. We consider a manifold $M$ with finitely many ends and a bgd-function $v$ that is doubling. In such cases, we have the following corollary.

\begin{cor}

If there exists a Grimaldi-Pansu metric on a manifold $M$ with finitely many ends such the volume growth function lies in the growth class of $v$ and satisfies the volume doubling condition, then $v(n)\leq Cn+C$ for some constant $C$.

\end{cor}

\begin{proof}
    This follows from Theorem \ref{doubling} on setting $\alpha=0$ in (\ref{4.1}). This is equivalent to the assumption that the function satisfies the volume doubling condition.
\end{proof}

\begin{ex}
    Consider the infinite rooted tree $T$ with each vertex having degree $2$, that is, the tree with two ends and no finite branches. For any finite set $\mathcal{U}$ of closed manifolds, we can construct an open manifold $M_T$ by taking connected sums along $T$. For an element $P\in \mathcal{U}$, remove two balls from $P$ and put a Riemannian metrics of bounded geometry on the resultant manifold such that it is a product metric on some tubular neighbourhood of the boundaries. Use these metrics to construct a metric on $M_T$ as defined in Theorem \ref{graph1}. Then the volume growth function of $M_T$ is a linear function and satisfies the volume doubling condition. 
\end{ex}

In the more general case of manifolds with infinitely many ends, there are restrictions to the existence of a Grimaldi-Pansu metric satisfying certain volume growth conditions at each point. For polynomial growth, we have the following corollary.

\begin{cor}
    For a given manifold $M$, there does not exist a Grimaldi-Pansu metric satisfying $Ar^n\leq vol(x,r)\leq Br^{n+\alpha}$ for all $x\in M$ and $r>r_0$, for any $n>1$ and $0<\alpha< \frac{n-1}{2}$, for any constants $r_0, A,$ and $B$.
\end{cor}

\end{document}